\documentclass[12pt]{amsart}
\usepackage{graphicx}
\usepackage{amsmath}
\usepackage{amssymb}
\usepackage{url}
\usepackage{xy}
\xyoption{all}
\usepackage{soul}
\usepackage{psfrag}

\def\co{\colon\thinspace}

\newcommand{\p}{\partial}
\newcommand{\N}{\mathbb N}
\newcommand{\Z}{\mathbb Z}

\newcommand{\C}{\mathbb C}
\newcommand{\CC}{\mathcal C}
\newcommand{\CP}{\mathbb{CP}}
\newcommand{\CPbar}{\overline{\mathbb{CP}}}

\newcommand{\R}{\mathbb R}
\newcommand{\SW}{\textrm{SW}}
\newcommand{\DT}{{\bf T}}
\newcommand{\T}{\mathcal T}
\newcommand{\Tmn}{\T_{mn}}
\newcommand{\zm}{\Z_m}
\newcommand{\zn}{\Z_n}

\newcommand{\la}{\langle}
\newcommand{\ra}{\rangle}

\def\wtX{\widetilde{X}\rule{0pt}{3.1mm}}
\def\tSigma{\widetilde{\Sigma}}
\def\spinc{\textrm{Spin}^c}
\newcommand{\sT}{\mathbb{T}}
\newcommand{\ssw}{\mathcal{SW}}
\newcommand{\ti}{\;\;\makebox[0pt]{$\top$}\makebox[0pt]{$\cap$}\;\;}
\renewcommand{\phi}{\varphi}

\newtheorem{theorem}{Theorem}[section]
\newtheorem{thm}{Theorem}

\newtheorem{lemma}[theorem]{Lemma}

\newtheorem{proposition}[theorem]{Proposition}
\newtheorem{corollary}[theorem]{Corollary}%
\theoremstyle{definition}
\newtheorem{definition}[theorem]{Definition}
\newtheorem{remark}[theorem]{Remark}
\newtheorem{example}[theorem]{Example}

\title{Double point surgery and configurations of surfaces}
\author[Hee Jung Kim]{Hee Jung Kim}
\address{Department of Mathematical Sciences\\
SUNY-Binghamton\\
Binghamton, NY 13902-6000\\
USA
}
\email{\rm{heekim@math.binghamton.edu}}
\author[Daniel Ruberman]{Daniel Ruberman}
\address{Department of Mathematics, MS 050\newline\indent Brandeis
University \newline\indent Waltham, MA 02454}
\email{\rm{ruberman@brandeis.edu}}
\dedicatory{
Dedicated, with respect and admiration, to Jos\'e Maria Montesinos-Amilibia on the occasion of his $65^{th}$ birthday.
}

\thanks{The first
author was partially supported by  the Max-Planck-Institut f\"ur
Mathematik in Bonn, Germany, and the second author was partially
supported by NSF Grant 0804760.}

\begin{document}
\maketitle

\begin{abstract}
We introduce a new operation, double point surgery, on immersed surfaces in a $4$--manifold, and use it to construct knotted configurations of surfaces in many $4$--manifolds.  Taking branched covers, we produce smoothly exotic actions of $\zm \oplus \zn$ on simply connected $4$--manifolds with complicated fixed-point sets.
\end{abstract}

\section{Introduction}
A configuration of surfaces~\cite{gilmer:thesis} is an immersion with isolated singularities, of a possibly disconnected closed surface $\Sigma$ in a $4$--manifold $X$.  In this paper we consider configurations whose components are embedded, either in the smooth or topological category.  We introduce {\em double point surgery}, a variation of  the Fintushel-Stern rim surgery~\cite{fs:surfaces,fs:addendum}, and use it to create configurations that are smoothly knotted, without changing the topological type or the smooth embedding type of the individual components of the configuration.   Rim surgery is the same as a knot surgery~\cite{fs:knots} along a torus lying in boundary of the tubular neighborhood of a circle embedded in a surface.  Double point surgery is a knot surgery (or more generally, a twisted knot surgery~\cite{kim:surfaces}) along a torus lying in a neighborhood of an intersection of the components.

Using twisted double point surgery, we will prove:
\begin{thm}\label{t:new-config}
There is a family of smooth $2$-component configurations $(X,\Sigma_1^{(n)} \cup \Sigma_2^{(n)})$ such that
\begin{enumerate}
\item\label{t:smooth} $(X,\Sigma_1^{(m)} \cup \Sigma_2^{(m)})$ and $(X,\Sigma_1^{(n)} \cup \Sigma_2^{(n)})$ are smoothly inequivalent for $m \neq n$.
\item\label{t:sigma1} $(X,\Sigma_1^{(m)})$ and $(X,\Sigma_1^{(n)})$ are smoothly equivalent for all $m,n$.
\item\label{t:sigma2} $(X,\Sigma_2^{(m)})$ and $(X,\Sigma_2^{(n)})$ are smoothly equivalent for all $m,n$.
\item\label{t:top} $(X,\Sigma_1^{(m)} \cup \Sigma_2^{(m)})$ and $(X,\Sigma_1^{(n)} \cup \Sigma_2^{(n)})$ are topologically equivalent for all $m,n$.
\end{enumerate}
\end{thm}
\begin{remark}
There are several equivalence relations that one might consider among configurations: ambient diffeomorphism or homeomorphism, as well as the {\em a priori} stronger relations of smooth or topological ambient isotopy.  The statements in the preceding theorem refer to equivalence up to diffeomorphism or homeomorphism preserving the orientation of the ambient manifold.  One could also restrict to equivalences preserving orientations on the components of the configuration; see Section~\ref{S:smooth} for a discussion.   In the topological category, homeomorphisms can often be promoted to isotopies using the work of Perron~\cite{perron:isotopy2} and Quinn~\cite{quinn:isotopy}.  In the smooth category, we are generally able to prove equivalence up to diffeomorphism; it seems reasonable to expect that this could be strengthened to isotopy but we do not know how to provide such an improvement.
\end{remark}

In section~\ref{S:actions}, we apply the double point surgery to the construction of group actions, in the fashion of a recent paper of Fintushel-Stern-Sunukjian~\cite{fintushel-stern-sunukjian:actions}.
\begin{thm}\label{t:exotic-action}
Let $m$ and $n$ be relatively prime numbers. There is a
simply--connected $4$--manifold $\widetilde{X}$ supporting infinitely many
$\zm \oplus \zn$-actions that are smoothly inequivalent, but
topologically equivalent. 
\end{thm}
Since $(m,n) = 1$, the group $\zm \oplus \zn$ is isomorphic to $\Z_{mn}$; writing the group as $\zm \oplus \zn$ is convenient for keeping track of the isotropy subgroups of the action.  In particular, since the singular set consists of a configuration, rather than an embedded surface, these examples are different from those presented in~\cite{fintushel-stern-sunukjian:actions}. In the terminology of group actions~\cite{bredon}, the actions constructed in~\cite{fintushel-stern-sunukjian:actions} are semi-free, whereas ours are not.
\\[1ex]
{\bf Conventions:}
Let $\Sigma$ be a surface that is the union of finitely many components $\Sigma_1,\ldots,\Sigma_n$, and let $f \co \Sigma \to X$ be an immersion.  Then $\Sigma$ will also denote the image of $f$, and the images of the $\Sigma_i$ will be referred to as the components of the configuration.   $\nu(\Sigma)$ will denote a regular neighborhood of a configuration, or a tubular neighborhood of a surface.     We will refer to the fundamental group of the complement as the group of the configuration.  All surfaces will be oriented, and a meridian will mean an oriented meridian.  A torus in a $4$--manifold on which a torus surgery is to be performed will be written in bold $\DT$, and other tori will be written in a normal typeface $T$.
\\[1ex]
\textbf{Acknowledgement.} We thank Inan\c{c} Baykur for his help with the details of the construction in Proposition~\ref{P:tori}, and Paul Melvin for some helpful correspondence.
\section{Double point surgery}\label{S:dsurgery}
Let us first recall the Fintushel-Stern knot surgery~\cite{fs:knots} along a torus $\DT$ embedded with trivial normal bundle in a $4$--manifold $X$.  Given a knot $K$ in $S^3$, we perform Fintushel-Stern knot surgery along $\DT$ by removing a neighborhood of $\DT$ and gluing the exterior $E(K)$ of $K$ times with $S^1$ via a diffeomorphism $\phi$ of $T^3$:
\[X_{K,\phi} = X_K=X-\DT \times D^2\cup_{\phi} E(K)\times S^1.\]
We will always assume that the gluing map $\phi$ has the following two properties:
\begin{enumerate}
\item $\phi$ identifies a meridian $ \p D^2$ of $\DT$ with a longitude $\lambda_K$ of $K$.
\item $\phi$ is {\em local} in the sense that for $\DT$ standardly embedded in $B^4$, there is a diffeomorphism of $B^4_{K,\phi}$ with $B^4$ that is the identity on the boundary.
\end{enumerate}
As described below, work of Plotnick~\cite{plotnick:fibered} gives many examples of local gluing maps.
As in~\cite{fs:surfaces}, the image of a surface $\Sigma$ embedded in $X - \DT$ under a diffeomorphism $X_K \to X$ may well be different from $\Sigma$. We will denote a surface constructed in this fashion by $(X,\Sigma_{K,\phi})$ or $(X,\Sigma_K)$ if $\phi$ is unspecified.

A local model for a transverse intersection of surfaces in a $4$--manifold is the intersection of the coordinate axes in $\C^2$.  The {\em double point torus} $\DT$ (linking torus, in the terminology of~\cite{freedman-quinn}) is the set
\[\{(z,w)\; | \; |z| = |w| = 1/2 \}  \subset \mathit{int}(B^4) \subset \C^2.
\]
Let $K$ be a knot in $S^3$ and let $\phi$ be a local gluing map.  Let $B_1, B_2$ be the unit disks in the two coordinate axes; then we can write $\DT =\mu_{1}\times \mu_{2}$ where $\mu_i$ is a meridian of $B_i$.
\begin{definition}\label{d:surgery}
Let $K$ be a knot in $S^3$.  A local double point surgery with local gluing map $\phi$ replaces  $(B^4,B_1 \cup B_2)$ by $(B^4,(B_{1}\cup B_{2})_{K,\phi})$.
\end{definition}
By assumption, the boundary (which is $S^3$, containing a Hopf link) is unaffected by this operation.  Hence we can do a local double point surgery in a neighborhood of a transverse intersection of surfaces; the result will be called a double point surgery on a configuration.

We will now consider double point surgery in greater detail.  Let $X$ be an oriented simply--connected $4$--manifold and let $\Sigma$ be a configuration in $X$. For convenience, we may assume that the configuration $\Sigma$  has two components $\Sigma_1$ and $\Sigma_2$, and the surgery is being done at an intersection point of $\Sigma_1$ and $\Sigma_2$.
Because the double point surgery operation is local, it may be written as follows.  Denoting by $(X',\Sigma')$ the complement of the $4$-ball neighborhood
$(B^4, B_1 \cup B_2)$ of a double point in $(X,\Sigma)$, we have
\begin{equation}\label{dbleptsurgery}
(X,\Sigma_K)=(X',\Sigma')\cup_{(S^3,\p B_1\cup \p B_2)}(B^4, (B_1 \cup B_2)_{K}).
\end{equation}

A useful observation is that local double point surgery is
closely related to surgery on twins, introduced by
Montesinos~\cite{montesinos:twins.I,montesinos:twins.II} and
Plotnick~\cite{plotnick:fibered}, described as follows. Let $P$ be a
plumbing at two points (with opposite signs) of two copies of $S^2\times D^2$. The pair
$S_1,\ S_2$ of cores of the copies of $S^2\times D^2$ is called a
`twin' by Montesinos~\cite{montesinos:twins.I,montesinos:twins.II}.  It has a standard embedding in $S^4$, with each component an unknotted $2$-sphere.  
Note that $\partial P$ is a $3$-torus, and we choose a basis $\{e_1,
e_2, e_3\}$ of $H_1(\p P)$ as follows. Let $e_1$ be the homology
class of the meridian $\mu_1$ of $S_1$, $e_2$ the homology class of the
meridian $\mu_2$ of $S_2$, and $e_3$ the homology class of a canonical curve
on $\partial P$ which generates $H_1(P)$. We note that $S^4 = P\cup_{\p}  \DT \times D^2$ where $\DT$ is a torus of the
form $\mu_1\times \mu_2$. Plotnick modified the pair $(S^4, P)$ by gluing
$E(K)\times S^1$ to $P$ using an identification of $\p P$ with $\p
E(K)\times S^1$, which is encoded by a matrix $A$ with respect to
the ordered basis $\{e_1, e_2, e_3\}$ of $H_1(\p P)$ and an ordered
basis $\{ \mu_{K}, S^1, \lambda_K\}$ of $H_1(\p E(K)\times S^1)$.
So, we have
\begin{equation}\label{plotnick1}
(S_1\cup S_2)_{K,A} \subset P\cup_{A} E(K)\times S^1
\end{equation}
where the general form of $A$ is
\begin{equation}\label{matrix1}
   \begin{pmatrix}
      p&k&0\\
      -\gamma&\beta&0\\
      -\alpha\gamma+bp&\alpha\beta+bk&1
\end{pmatrix}  \hspace{1cm}           p\beta+k\gamma=1.
\end{equation}

For some choices of the parameters (see~\cite[Corollary 6.1]{plotnick:fibered}), the gluing map is local and the construction produces $(S^4,(S_1\cup S_2)_{K,A})$. The torus $\DT$ is exactly the double point torus for $S_1 \cup S_2$ near one of the double points of the twin.   This implies:
\begin{proposition}\label{P:embedd-plotnick} The double point surgered pair $(B^4, (B_1 \cup B_2)_{K})$ is diffeomorphic
to the punctured pair $(S^4, (S_1\cup S_2)_{K,A})$.
\end{proposition}
The decomposition of $(X,\Sigma_K)$ in~\eqref{dbleptsurgery} and
Proposition~\ref{P:embedd-plotnick} thus allow us to write $(X,\Sigma_K)$ as
a connected sum;
\begin{equation}\label{E:dsurgery-twin}
(X,\Sigma_K)=(X,\Sigma)\sharp (S^4,(S_1\cup S_2)_{K,A}).
\end{equation}

By analogy with `twist-rim surgery' in~\cite{kim:surfaces}, we are particularly interested in 
double point surgery performed using the gluing map $\phi_k$ defined
by
\begin{equation}\label{twistbleptsurgery}
\phi_{k*}(\mu_1)=\mu_K,\  \phi_{k*}(\mu_2)=k\mu_K+[S^1],\ 
\text{and}\ \phi_{k*}[\p D^2]=\lambda_K.
\end{equation}
Locally, this corresponds to the knotted twin described by the matrix
\begin{equation}\label{E:matrix2}
A(k)=\begin{pmatrix}
         1&k&0\\
         0&1&0\\
         0&0&1
   \end{pmatrix}
\end{equation}
which is in fact a local gluing.  In $S^4$, the sphere $S_1$ becomes Zeeman's~\cite{zeeman:twist} $k$-twist spin of $K$, denoted $A(K,k)$, while the sphere $S_2$ is unaffected by the surgery.
We refer to this surgery as a `$k$-twisted double point surgery'.

Let us denote by $\Sigma_{1,K}$ and $\Sigma_{2,K}$ the components of
the new configuration $\Sigma_K$.  According to proposition~\ref{P:embedd-plotnick} and equation~\eqref{E:dsurgery-twin}, the result of $k$-twisted double point surgery on the configuration $\Sigma$ is given as follows.
\begin{corollary}\label{embedding}
Suppose that  $(X, \Sigma_K)$ is the pair obtained by a $k$-twisted
double point surgery. Then the embedding of $\Sigma_{1,K}$ is
diffeomorphic to $\Sigma_1\sharp A(K,k)$ whereas the embedding of
$\Sigma_{2,K}$ is equal to $ \Sigma_2$.
\end{corollary}

\section{Fundamental Group}\label{S:pi1}

In this section, we study the effect of double point surgery on the group of a configuration $\Sigma$ in a simply--connected $4$--manifold $X$. From the decomposition of~\eqref{dbleptsurgery}, the complement of the configuration in $X$
can be written
\begin{equation}\label{knotsurgerycompl}
X-\Sigma_{K}=(X'-\Sigma')\cup_{T\times I}(B^4-(B_1 \cup B_2)_{K}).
\end{equation}
Here $T\subset \partial (B^4-(B_1 \cup B_2)_{K}) $ is parallel to the double point torus $\DT$.
Since $\pi_1(X'-\Sigma')$ is isomorphic to $\pi_1(X-\Sigma)$, we
need only compute the group of the nonstandard
component in this decomposition.
\begin{lemma}\label{L:fungp1} For any gluing $\phi$ and any knot $K$,
$\pi_1(B^4-(B_1 \cup B_2)_{K})$ is isomorphic to $\pi_1(E(K)\times S^1)$.
\end{lemma}

\begin{proof}
Since $B^4 - \nu(B_1 \cup B_2)$ is diffeomorphic to $\DT \times D^2$, it follows that $B^4-(B_1 \cup B_2)_{K}$ is actually diffeomorphic to $E(K) \times S^1$.
\end{proof}
Note that the gluing map $\phi$ is encoded via the composition
\[
\pi_1(S^3 - \partial(B_1 \cup B_2)) \to \pi_1(B^4-(B_1 \cup B_2)_{K}) \overset{\cong}{\longrightarrow} \pi_1(E(K)\times S^1).
\]

We shall study conditions under which the group of a configuration $\Sigma$ in $X$ is preserved by a double point surgery.   This is true, for instance, when the complement of $\Sigma$ is simply--connected.  More generally, we show that for particular gluing maps, some abelian fundamental groups are preserved.

The basic tool for this is the following diagram to be used in applying the Van Kampen theorem to the decomposition  of $X-\Sigma_K$ in~\eqref{knotsurgerycompl}:
\begin{equation}\label{d-diagram}
\xymatrix@C=0pc{
&\pi_1(B^4-(B_1 \cup B_2)_{K})Ê\ar[dr]^{j_1}&\\
{\pi_1 (T\times I)} Ê\ar[ur]^{i_1} \ar[dr]^{i_2}
&&\pi_1(X - \Sigma_{K})\\
&\pi_1(X'-\Sigma')Ê\ar[ur]^{j_2}&
}
\end{equation}
To make the diagram and subsequent calculations easier to read, we use the same notation for a map and for the homomorphism that induces on the fundamental group.

\begin{proposition}\label{P:fundamentalgp} Suppose that $X$ is simply-connected and that $\pi_1(X - \Sigma)$ is abelian. Then $i_2$ in diagram~\eqref{d-diagram} is surjective.  Moreover, the group is preserved by $k$-twisted double point surgery in the following cases:
\begin{enumerate}
\item\label{F1}
Suppose that $\pi_1(X-\Sigma)$ is an infinite cyclic group
generated by the meridian $\mu_2$ of $\Sigma_2$. If $k=0$ then $\pi_1(X -
\Sigma_{K}) \cong \Z$.
\item \label{F2}
Suppose that  $\pi_1(X-\Sigma)$  is a finite cyclic group $\Z_q$
generated by the meridian $\mu_1$ of $\Sigma_1$ with a relation
$\mu_1^p\mu_2=1$ for some $p$. If $(p+k,q)=1$, then $\pi_1(X -
\Sigma_{K}) \cong \Z_q$.
\item\label{F3}
Suppose that $\pi_1(X-\Sigma)$ is $\zm \oplus \zn$ so that the
meridian $\mu_1$ of $\Sigma_1$ is a generator of $\zm$ and the other
meridian $\mu_2$ of $\Sigma_2$ is a generator of  $\zn$. If $(m,
kn)=1$, then $\pi_1(X - \Sigma_{K})$ is $\zm \oplus \zn$.
\end{enumerate}
\end{proposition}
\begin{proof}
In the diagram~\eqref{d-diagram} given by the Van Kampen Theorem,
$\pi_1 (T\times I)$ is $\Z\oplus\Z$ generated by $\mu_1$, $\mu_2$
and $\pi_1(X'-\Sigma')\cong \pi_1(X-\Sigma)$, which is normally generated by $\mu_1$, $\mu_2$. Hence if
$ \pi_1(X-\Sigma)$ is abelian, $i_2$ is onto.
This implies that in all three cases of the proposition, $\pi_1(X - \Sigma_{K})$ is isomorphic to
$\pi_1(B^4-(B_1 \cup B_2)_{K})Ê/ \la i_1(\ker i_2)\ra$.

Case~\ref{F1}.  Since $\pi_1(X-\Sigma) \cong \Z$, we must have a relation
$\mu_1\mu_2^d =1$ for some $d\in\Z$.  The element $\mu_1\mu_2^d$ generates $\ker i_2$ and
its image under $i_1$ is
$\mu_K(\mu_K^k[S^1])^d$.  Using Lemma~\ref{L:fungp1}, the group
$\pi_1(B^4-(B_1 \cup B_2)_{K})Ê/ \la i_1(\ker i_2)\ra$ has a
presentation $\la \pi_1(E(K)\times S^1) \mid
\mu_K(\mu_K^k[S^1])^d=1\ra$ which is the same as 
\[
\la \pi_1(E(K)),[S^1]\mid \mu_K^{1+kd}[S^1]^{d}=1, [\beta,
[S^1]]=1,\forall\beta\in\pi_1(E(K)) \ra.
\]
So, if $k=0$, then
$[S^1]^{d}=\mu_K^{-1}$. Since $[S^1]$ commutes with every element in
$\pi_1(E(K))$, so does $\mu_K$. This means that the group $\pi_1(X -
\Sigma_{K})$ is abelian, and hence isomorphic to $\Z$.

Case~\ref{F2}.
In this situation, $\ker i_2$ is generated by
$\mu_1^q$ and $\mu_1^p\mu_2$, and so the group $\pi_1(X - \Sigma_{K})$
has a presentation 
\[ \la \pi_1(E(K)), [S^1] \mid \mu_K^q=1, \ \mu_K^p(\mu_K^k[S^1])=1,  [\beta,
[S^1]]=1,\forall\beta\in\pi_1(E(K)) \ra.
\]
If $(p+k,q)=1$ then
$(p+k)s+qt=1$ for some $s$, $t$. If we raise the 
relation $\mu_K^p(\mu_K^k[S^1])=1$ to the $s^{th}$ power, we get $(\mu_K^p(\mu_K^k[S^1]))^s=\mu_K^{(p+k)s}[S^1]^s=1$ (note that
$\mu_K$ and $[S^1]$ commute), and moreover
$\mu_K^{(p+k)s}[S^1]^s=\mu_K^{1-qt}[S^1]^s=\mu_K[S^1]^s=1$. This
relation makes $\mu_K$ commute with every element in $\pi_1(E(K))$
and so $\pi_1(X - \Sigma_{K})$ is $\Z_q$.

Case~\ref{F3}.
As above, $\ker i_2$ is generated by $\mu_1^m$,
$\mu_2^n$ and so the presentation of $\pi_1(X - \Sigma_{K})$ is 
\[
\la \pi_1(E(K)), [S^1] \mid \mu_K^m=1, (\mu_K^{k}[S^1])^{n}=1,
[\beta, [S^1]]=1,\forall\beta\in\pi_1(E(K)) \ra.
\]
If $(m,kn)=1$
then $ms+knt=1$ for some $s$, $t$. We note the relation
$1=((\mu_K^{k}[S^1])^{n})^t=\mu_K^{knt}[S^1]^{nt}=\mu_K^{1-ms}[S^1]^{nt}=\mu_K[S^1]^{nt}$,
which makes $\mu_K$ commute with every element in $\pi_1(E(K))$.
Thus, the result follows.
\end{proof}

In the next section, we will use configurations corresponding to the first two cases to construct examples that will prove Theorem~\ref{t:new-config}.  The third case will be used in proving Theorem~\ref{t:exotic-action}.

\section{Examples}\label{S:examples}
In the previous section we determined the effect of a double point surgery on the group of a configuration.  In this section we give some examples of configurations such that the group is abelian.  Combined with the calculations in the previous sections, we find double point surgeries with appropriately chosen gluing $\phi$, for which the group is unchanged.

A first observation is the following lemma, which computes the homology of the complement of a configuration.
\begin{lemma}\label{L:H1}
Suppose $X$ is a $4$--manifold with $H_1(X) = 0$, and suppose that $\Sigma$ is a configuration of surfaces $\Sigma_1,\ldots,\Sigma_k$.   Choose a basis $A_1,\ldots,A_n$ for $H_2(X;\Z)$.  Then $H_1(X-\Sigma)$ has a presentation with generators the meridians $\mu_i$ of the components of $\Sigma$, and one relation
\begin{equation}\label{E:relation}
\Sigma_{j=1}^k a_{ij} \mu_j = 0
\end{equation}
for each generator $A_i$, where $a_{ij}= A_i \cdot \Sigma_j$.
\end{lemma}
\begin{proof}
By standard arguments, the meridians $\mu_j$ generate $H_1(X - \Sigma)$.  Let $A\in H_2(X)$; then if we intersect an embedded surface representing $A$ with $\Sigma$ and remove a regular neighborhood of $\Sigma$ we see a relation of the form~\eqref{E:relation}
where $a_j = A \cdot \Sigma_j$. Conversely, given such a relation, we build a surface $A$ as the union of meridian discs to $\Sigma$, together with a surface representing a null-homology for $\Sigma_{i=1}^k a_i \mu_i$.   Note that if $C = A + B$, then the relation for $C$ is the sum of the relations for $A$ and $B$.  Thus a basis for $H_2(X)$ gives rise to the desired presentation.
\end{proof}

Configurations with trivial group are easy to come by.  For any configuration $\CC = (C_1,\ldots,C_k)$ in a simply--connected manifold $X$,  the fundamental group
$\pi_1(X - \CC)$ is generated by the meridians of the $C_i$.  Note that if we blow up $X$ along some $C_i$, then the meridian of the new surface $C_i$ becomes trivial, because $C_i$ intersects the exceptional curve in one point.  Hence, given any configuration $\CC$, we can create a new configuration $\CC' \subset X' = X \sharp_k \CPbar^2$ with $\pi_1(X' - \CC') $ trivial by blowing up $k$ times, once on each $C_i$.

More interesting examples of abelian fundamental groups come from the Zariski conjecture, proved by Deligne~\cite{deligne:zariski} and Fulton~\cite{fulton:zariski}, which states that the complement of a nodal curve in $\CP^2$ has abelian fundamental group.   This implies the following.
\begin{example}\label{E:nodal}
Let $\Sigma_1,\Sigma_2$ be smooth complex curves in $\CP^2$ in general position, having degrees $d_1, d_2$ respectively.  Then $\pi_1(\CP^2 - (\Sigma_1\cup \Sigma_2)) \cong \Z \oplus \Z_d$ where $d= \gcd(d_1, d_2)$.  According to case~\ref{F1} of Proposition~\ref{P:fundamentalgp}, if $d_1=1$, the operation of $0$-twisted double point surgery gives a new configuration whose group is $\Z$.
\end{example}

A second example is provided by a generalization of the Zariski conjecture to other algebraic surfaces, due to Nori~\cite{nori:zariski}.
\begin{example}\label{E:rational}
Let $\Sigma_1,\Sigma_2$ be smooth complex curves in general position in $X ={\bf P}^1 \times {\bf P}^1$ carrying the homology classes  $(p,q)$ and $(1,0)$ respectively.  Assume that $p$ and $q$ are both positive.  Note that $X - \nu(\Sigma_2) $ is simply connected, being diffeomorphic to $S^2 \times D^2$, and that $\Sigma_1\cdot \Sigma_1 > 0$.  Hence theorem II of~\cite{nori:zariski} applies to show that the fundamental group of $X - (\Sigma_1\cup \Sigma_2)$ is abelian.  Lemma~\ref{L:H1} yields the relations (written additively)
\begin{align*}
q\, \mu_1 & = 0 \\
p\, \mu_1 + \mu_2 & = 0.
\end{align*}
So the group $\pi_1(X - (\Sigma_1\cup \Sigma_2)) = H_1(X - (\Sigma_1\cup \Sigma_2)) \cong \Z_q$, generated by $\mu_1$. Written multiplicatively, the second relation  between the meridians reads $\mu_1^p\mu_2=1$. Thus we may do a $k$-twisted double point surgery on the configuration $\Sigma_1 \cup \Sigma_2$, using any knot $K$, and $k$ satisfying $(p+k,q)=1$.  By case~\ref{F2} of Proposition~\ref{P:fundamentalgp} this gives a new configuration  with group  $ \Z_q$.
\end{example}

\begin{example}\label{E:spheres}  We will construct  connected surfaces  $\Sigma_m$ and $\Sigma_n$ in $X = S^2 \times S^2$ carrying the classes $(m,0)$ and $(0,n)$ respectively, with $\pi_1(X - \Sigma)\cong \zm \oplus \zn$.  Start with the configuration $S_m \cup S_n$, where $S_m$ consists of $m$ parallel copies of $S^2 \times \{q\}$ and $S_n$ consists of $n$ parallel copies of $ \{p\} \times S^2$.  The complement is a product $\left( S^2 - \{p_1,\ldots,p_n\} \right) \times \left( S^2 - \{q_1,\ldots,q_m\} \right)$ and so has fundamental group
\begin{equation}\label{e:smn}
\langle \mu_1,\ldots \mu_m, \nu_1,\ldots \nu_n\;|\; [\mu_i,\nu_j] =1;\ \prod \mu_i=1=\prod\nu_j \rangle
\end{equation}

The configuration $(S^2 \times S^2,S_m \cup S_n)$ may be visualized as follows (compare~\cite{akbulut-kirby:branch}). Take the standard handle decomposition of $S^2 \times S^2$, with two $0$-framed $2$-handles added to $B^4$ along a Hopf link.  Then $S_m$ is $m$ copies of the core of one of the $2$-handles, together with $m$ disjoint disks in the $4$-ball; $S_n$ has a similar description.  To make the surface $\Sigma_m$, connect up the  boundaries of the $m$ copies of the $2$-handle core with $m-1$ oriented $1$-handles.  The result is an unknotted circle, which may be filled in by a disk in the $4$-ball.  Do a similar construction to make $\Sigma_n$. This operation is depicted in Figure~\ref{F:smn}, for $m=3$ and $n=2$.
\begin{figure}[!ht]\label{F:smn}
\centering
\psfrag{m1}{$\mu_1$}
\psfrag{mm}{$\mu_m$}
\psfrag{n1}{$\nu_1$}
\psfrag{nn}{$\nu_n$}
\psfrag{ra}{$\Longrightarrow$}
\includegraphics[scale=1]{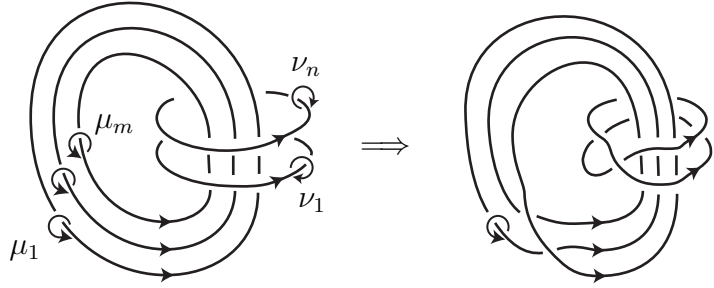}
\caption{Spheres in $S^2 \times S^2$}
\label{F:spheres}
\end{figure}
The $1$-handle addition that connects up the $j^{th}$ and $k^{th}$ disks in $S_m$ adds the relation $\mu_j = \mu_k$ to the presentation~\eqref{e:smn}, and similarly all of the meridians of the components of $S_n$ become equal when we pass to $\Sigma_n$.  It follows that  $\pi_1(X - \Sigma)\cong \zm \oplus \zn$, with the meridians of $\Sigma_m$ and $\Sigma_n$ generating the two summands.  If $(m,kn) =1$, then applying case~\ref{F3} of Proposition~\ref{P:fundamentalgp} gives a new configuration whose group is $ \zm \oplus \zn$.
\end{example}

\subsection{A symplectic configuration}\label{ss:symp}
In the next section, we will show that the configurations constructed in examples~\ref{E:nodal} and~\ref{E:rational}
can be smoothly knotted by double point surgery, for appropriate choices of the knot $K$.  This will be done by
showing that the smooth surfaces gotten by smoothing all of the double points are knotted.  This strategy will succeed,
for instance, if the configuration is symplectically embedded with all intersections positively oriented and the Alexander
polynomial of $K$ is non-trivial.  Unfortunately, this technique will not work to show that example~\ref{E:spheres}
yields smoothly knotted configurations, because of the observation of
Tian-Jun Li~\cite[Example 3.3]{li:symplectic-surfaces} that the adjunction formula prevents the class $(m,0)$ in $H_2(S^2 \times S^2)$ from being represented by
a connected symplectic surface.

On the other hand, smoothly knotted configurations whose complements have $\pi_1 = \zm \oplus \zn$ are a key
ingredient in the proof of Theorem~\ref{T:exotic-action}.   The next example gives a symplectic configuration with this
group, using a more complicated construction than that in example~\ref{E:spheres}.

\begin{proposition}\label{P:tori}
For any positive $m$ and $n$, there is a simply--connected symplectic $4$--manifold $X$, containing a symplectic configuration $\T_{mn} =  T_m
\cup T_n$ with each $T_i$ a torus, and
\[\pi_1(X - \T) \cong \zm \oplus \zn\]
generated by the meridians of $T_m$ and $T_n$.
\end{proposition}
\begin{proof}
Consider the $4$-torus $T^4 = S_1 \times S_2 \times S_3 \times S_4$, which contains the six standard 2-dimensional
tori $T_{ij}= S_i \times S_j$, each of which is either symplectic or Lagrangian.  We start with the symplectic
configuration $\T =T_{12} \cup T_{34}$, and want to do fiber
sums~\cite{gompf:symplectic,mccarthy-wolfson:symplectic-sum} with an elliptic surface $E(1)$ along
symplectic or Lagrangian tori missing $\T$ to make its complement simply
connected.   (Sums along Lagrangian tori require a preliminary deformation of the symplectic structure~\cite{gompf:symplectic}.)  If we set $M = \left(T^4 \#_{T_{13} = F} E(1)\right) \#_{T_{14} = F} E(1)$, then $\pi_1(M) \cong \Z$ generated by the circle $S_2$.   The other tori ($T_{24}$ and $T_{23}$) that miss $\T$ intersect either $T_{13}$ or $T_{14}$ and so cannot be used to perform a final fiber sum to kill this generator.

However, recall that $E(1)$ has a 2-sphere section, of square $-1$.  When we do the first fiber sum (along $T_{13}$) then we remove a disk from this section, and also from the torus $T_{24}$.   Thus, $M$ contains a symplectically embedded torus $T'$ of square $-1$, made from the connected sum of $T_{24}$ with the section in $E(1)$.   In $\CP^2$, there is a symplectic torus of square $9$, representing 3 times the generator, and hence $\CP^2$ blown up $8$ times contains a torus $T''$ of square $1$.  So we can make one more fiber sum, to get the manifold
\[
X = M \#_{T' = T''} \left(\CP^2\#_8 \CPbar^2\right).
\]
By construction, $X$ is symplectic, and contains the symplectic configuration $\T$.  The complement of $\T$ is simply--connected, because $F \subset E(1)$ and $T'' \subset \CP^2\#_8 \CPbar^2$ both have simply--connected complements.

\begin{figure}[!ht]
\centering
\includegraphics[scale=.7]{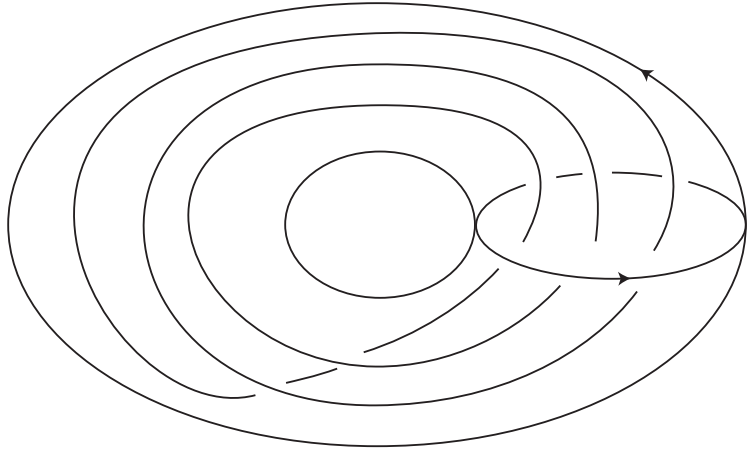}
\caption{Braid $k_m$ in $S^1 \times D^2$}
\label{F:braid}
\end{figure}
To build the configuration $\Tmn$, we need an observation of Fintushel-Stern~\cite{fs:symplectic-homology} (compare~\cite{moishezon:braids-chisini}) that in a neighborhood of a symplectic torus of square $0$, there is a `braided' symplectic torus representing $m$ times the generator of $H_2$, for any $m > 0$.  In fact, if the neighborhood is written as $S^1 \times (S^1 \times D^2)$, the braided torus can be taken to be $S^1$ cross the $(m,1)$-torus knot in $S^1 \times D^2$, representing $m$ times the generator of $H_1(S^1 \times D^2)$, and denoted $k_m$ in what follows.

Choose a disc $D_{12}$ lying in $T_{12}$, and a set $P$ of $n$ points in its interior, arranged in cyclic order as in Figure~\ref{F:points}; similarly choose $m$ points $Q$ in the interior of a disc $D_{34} \subset T_{34}$.  We will use a superscript ${}^*$ to indicate punctured objects, so that for instance $D_{12}^* = D_{12} - P$ and $T_{12}^* = T_{12} - \mathrm{int}(D_{12})$.
\begin{figure}[!ht]
\centering
\psfrag{c1}{$\mu_1$}
\psfrag{c2}{$\mu_2$}
\psfrag{c3}{$\mu_3$}
\psfrag{c}{$b$}
\psfrag{rc1}{$\rho(\mu_1)$}
\psfrag{rc2}{$\rho(\mu_2)$}
\psfrag{rc3}{$\rho(\mu_3)$}
\psfrag{r}{$\rho$}
\includegraphics[scale=.9]{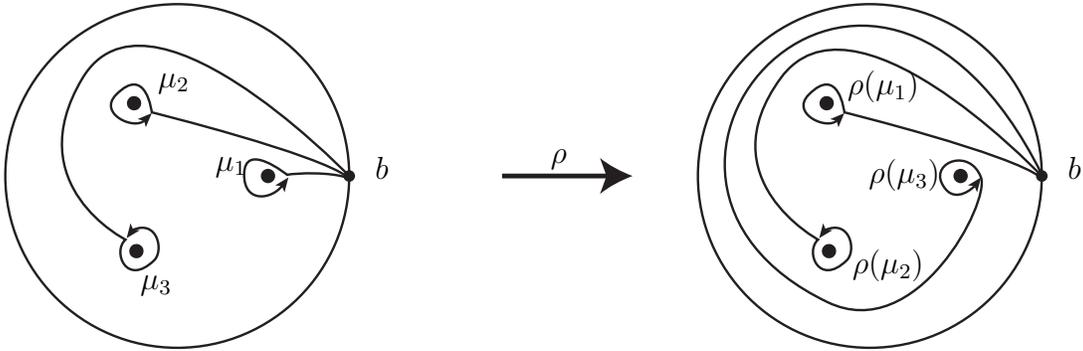}
\caption{$D_{12}^*$ and rotation $\rho_3$}
\label{F:points}
\end{figure}
Let $b$ be a point on $\partial D_{12}$ that will be used as a base point for both $D_{12}^*$ and $T_{12}^*$ and let
$c$ be a point on $\partial D_{34}$ that will be used as a base point for both $D_{34}^*$ and $T_{34}^*$.   The meridians of the points $P$, with base paths as indicated in Figure~\ref{F:points} will be named $\mu_i$, and the meridians of $Q$ will be called $\nu_j$.   The base point for all of the calculations of the $4$-dimensional  complements will be $b \times c$.

\begin{definition}\label{D:tmn}
\begin{enumerate}
\item The configuration $\Tmn$ is $T_m \cup T_n$ where the first torus $T_m$ is $S_1 \times k_m \subset S_1 \times (S_2 \times D_{34})$ and $T_n = k_n \times S_4 \subset (D_{12}\times S_3) \times S_4$.
\item
Let $\rho_m$ be a diffeomorphism of $D_{34}$ that permutes the
punctures cyclically, and is the identity on the boundary, and let
$\rho_n$ be the corresponding diffeomorphism of $D_{12}$.
\end{enumerate}
\end{definition}
\noindent
Note that  $T_m$ is $S_1$ times the mapping torus $S_2 \times_{\rho_m} Q \subset S_2 \times_{\rho_m} D_{34} = S_2 \times D_{34}$ with a similar description for $T_n$.

By construction, and Lemma~\ref{L:H1}, the homology group of the complement is $\zm \oplus \zn$; we now calculate its fundamental group.  The strategy is to divide the complement of $\Tmn$ into two pieces: the complement in $X$ of a regular neighborhood $N$ of $\T$, and $N - \Tmn$.  By construction, $X - N$ is simply--connected, so the main work goes into computing $\pi_1(N- \Tmn)$.

To describe $N$, recall that $T^4$ has a handle decomposition with a $0$-handle $H$, four $1$--handles $H_i$ and six $2$--handles $H_{ij}$, plus handles of index $3$ and $4$.  (Our convention is that the number of subscripts is the index of the handle.)    Then $N$ may be chosen to be the union of the $0$ and $1$--handles, plus the $2$--handles $H_{12}$ and $H_{34}$ with co-cores $D_{34}$ and $ D_{12}$.   We may assume that $\Tmn$ is embedded in $N$ so that $H$ is exactly $ D_{12} \times D_{34}$, and $\Tmn \cap H = D_{12} \times Q \cup P \times  D_{34}$. The complement of $\Tmn$ in $T^4$ is obtained by adding the rest of the handles to $N - \Tmn$, and then doing the fiber sums described above.

For the fundamental group calculation, choose curves $\alpha_1$  and $\alpha_2$ in $T_{12}^*$ that are homotopic (in $T^4$) to $S_1$ and $S_2$, according to the following recipe; the symbols refer to Figure~\ref{F:puncture}.  The loops $\alpha_1$ are given by the products of arcs
\[
\alpha_1 = a_1 \ast \gamma_3 \ast \gamma_4\quad\textrm{and}\quad \alpha_2 = \gamma_1 \ast a_2 \ast \gamma_4
\]
respectively.  The boundary of $D_{12}$, oriented as in the picture, is then given by $\gamma = \gamma_1 \ast \gamma_2 \ast\gamma_3\ast \gamma_4$.
\begin{figure}[!h]
\centering
\psfrag{a1}{$a_1$}
\psfrag{a2}{$a_2$}
\psfrag{c1}{$\gamma_1$}
\psfrag{c2}{$\gamma_2$}
\psfrag{c3}{$\gamma_3$}
\psfrag{c4}{$\gamma_4$}
\psfrag{b}{$b$}
\psfrag{h12}{$h_{12}$}
\includegraphics[scale=.9]{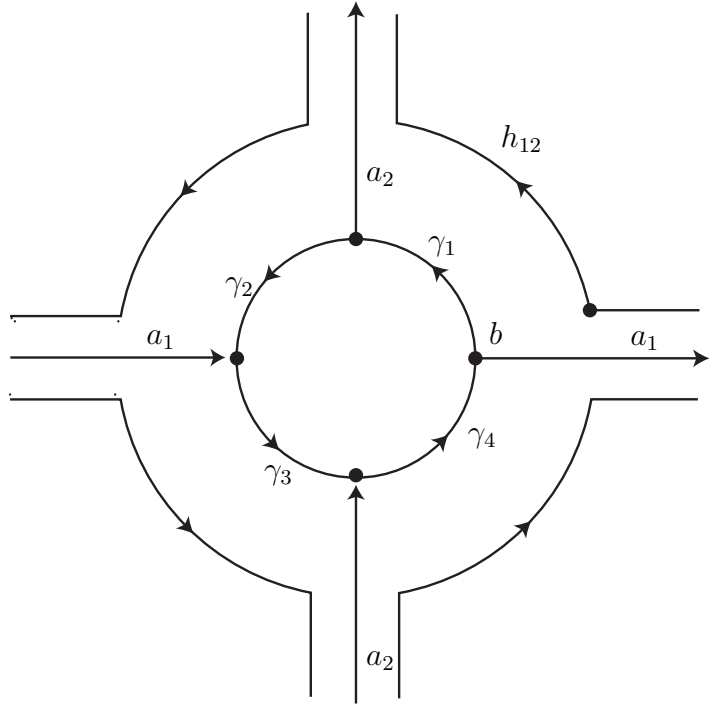}
\caption{Curves on $T_{12}^*$}
\label{F:puncture}
\end{figure}
The boundary of $D_{34}$ will be denoted $\eta$.

We collect some basic observations in a lemma.
\begin{lemma}  \label{L:pieces} $N- \Tmn = W \cup A \cup B$, where the pieces are described as follows.
\begin{enumerate}
\item\label{h0} $W = H - H \cap \Tmn = H - P \times D_{34} - D_{12} \times Q$ is the product $D_{12}^* \times D_{34}^*$.
\item The boundary of the configuration $P \times D_{34} \cup D_{12} \times Q $ in the $0$--handle $H$ is the $(n,m)$ cable of the Hopf link drawn in Figure~\ref{F:hopf} for $m=3,n=2$.
\begin{figure}[!ht]
\centering
\includegraphics[scale=.5]{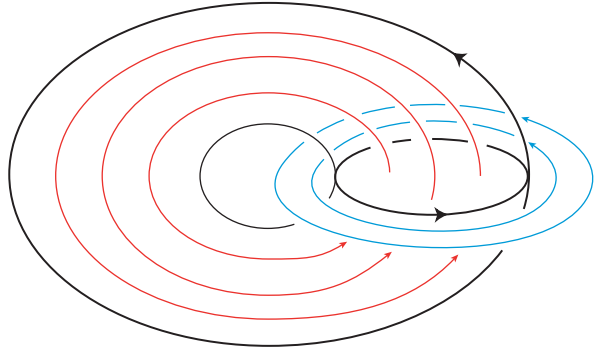}
\caption{$(m,n)$ cable of Hopf link}
\label{F:hopf}
\end{figure}
\item  $(N - \Tmn) - (H - H \cap \Tmn) $ is the union of two pieces $A$ and $B$, where $A$ is the complement of the part of $\Tmn$ lying in the $1$--handles $H_1,\; H_2$ and the $2$--handle $H_{12}$, and similarly for $B$.
\item\label{bundle12} $A$ is a flat $D_{34}^*$--bundle over $T^*_{12}$, with trivial monodromy around the circle $\alpha_1$ and monodromy around the circle $S_2$ given by $\rho_m$.
\item $B$ has a similar description as a flat $D_{12}^*$--bundle over $T^*_{34}$
\end{enumerate}
\end{lemma}

The description in Lemma~\ref{L:pieces}~\ref{bundle12} gives a presentation
\begin{multline*}
\pi_1(A,b\times c) = \langle \alpha_1,\alpha_2,\nu_1,\ldots,\nu_m, \gamma\; |\; \gamma = [\alpha_1,\alpha_2],\  [\gamma,\nu_i],\ [\alpha_1,\nu_i]\ (i= 1,\ldots, m),\\ \alpha_2 \nu_i \alpha_2^{-1} = \nu_{i+1}\ (i= 1,\ldots, m-1),\ \alpha_2 \nu_{m} \alpha_2^{-1} = (\nu_m\cdots \nu_{1})\cdot \nu_{1} \cdot (\nu_m\cdots \nu_{1})^{-1}
\rangle
\end{multline*}
The first relation (which makes the generator $\gamma$ redundant) comes from the addition of the $2$--handle $h_{12}$.

There is a similar presentation for $\pi_1(B,*)$, with generators $\alpha_3, \alpha_4, \mu_1,\ldots \mu_n$.   Finally, Lemma~\ref{L:pieces}~\ref{h0} implies that $\pi_1(W) \cong \pi_1(D_{12}^* \times D_{34}^*) = F_n \times F_m$, where $F_m,\ F_n$ are free groups generated by the $\mu_i$ and $\nu_j$ respectively.  Note that these generators all live in $\partial H \cap \Tmn$.  Two successive applications of van Kampen's theorem (adding $A$ and then $B$ to $W$) yield the following presentation of $\pi_1(N - \Tmn)$.   Unless specified to the contrary,  $i$ runs from $1$ to $m$  in any appearance of $\nu_i$ and $j$ runs from $1$ to $n$ in any appearance of $\mu_j$.
\begin{multline*}
\pi_1(N - \Tmn,b\times c) = \langle \alpha_1,\alpha_2,\alpha_3,\alpha_4, \nu_1,\ldots,\nu_m, \mu_1,\ldots,\mu_n\gamma, \eta\; |\; [\mu_j,\nu_i], \\
\gamma = [\alpha_1,\alpha_2],\ \gamma = \mu_n\cdots \mu_1,\ [\gamma,\nu_i],\\
 \eta = [\alpha_3,\alpha_4],\ \eta = \nu_m\cdots \nu_1,\  [\eta,\mu_j],\\
[\alpha_1,\nu_i]\ , \alpha_2 \nu_i \alpha_2^{-1} = \nu_{i+1}\ (i= 1,\ldots, m-1),\ \alpha_2 \nu_{m} \alpha_2^{-1} = (\nu_m\cdots \nu_{1})\cdot \nu_{1} \cdot (\nu_m\cdots \nu_{1})^{-1},\\
[\alpha_4,\mu_j]\ , \alpha_3 \mu_j\alpha_3^{-1} = \mu_{j+1}\ (j= 1,\ldots, n-1),\ \alpha_3 \mu_{n} \alpha_3^{-1} = (\mu_n\cdots \mu_{1})\cdot \mu_{1} \cdot (\mu_n\cdots \mu_{1})^{-1}
\rangle
\end{multline*}

When we add on $X- N$, all of the generators $\alpha_1,\ldots,\alpha_4$ are killed by the fiber sums. Moreover, since $X -N$ is simply connected, the curves $\gamma$ and $\eta$, which are the meridians of the original configuration $\T$, become trivial.  The last two rows of relations imply that $\nu_1= \cdots = \nu_m$ and $ \mu_1 = \cdots =\mu_n$, and the triviality of $\gamma$ and $\eta$ imply that $\mu_1^n =1 = \nu_1^m$. Since $\mu_1$ and $\nu_1$ commute, we finally get $\pi_1(X - \Tmn) = \zm \oplus \zn$, generated by the meridians of the components.
\end{proof}

\section{Topological classification}\label{S:top}
As described in Proposition~\ref{P:fundamentalgp}, in certain circumstances, $k$-twisted double point surgeries on configurations with group $\Z$ or $\Z_m$ preserve the group of the configuration. The main result of this section is that in fact the topological type of the configuration is also unchanged by the surgery.  As in Proposition~\ref{P:fundamentalgp}, we consider a $2$-component configuration $\Sigma_1\cup \Sigma_2$, with meridians $\mu_1, \mu_2$.
\begin{theorem}\label{T:topclass} In the three cases discussed in Proposition~\ref{P:fundamentalgp}, the configuration $\Sigma_{K, A(k)}$  resulting from $k$-twisted double point surgery is topologically isotopic to $\Sigma$.
\end{theorem}
\begin{proof}
In cases~\ref{F2} and~\ref{F3}, the group is finite cyclic, and the result follows directly from the main
results of our earlier paper~\cite[Theorems 1.2 and 1.3]{kim-ruberman:surfaces}.  These results use
the fact~\cite[section 11]{hambleton-taylor:guide}  that the Wall L-groups $L^{h}_5$ and $L^{s}_5$
vanish for a finite cyclic fundamental group.  In case~\ref{F1}, the L-groups $L^{h}_5(\Z)$ and
$L^{s}_5(\Z)$ are isomorphic to $\Z$, so a little additional work is required.
The outline of the proof is the same as in~\cite{kim-ruberman:surfaces}.

The first step is to construct a homotopy equivalence (rel boundary)
from $X - \nu(\Sigma_K) \to X -\nu(\Sigma)$.  This is done by
choosing a degree-1 homology equivalence between $E(K) \times S^1$
and $\nu(\DT)$, where $\DT$ is the double point torus that is the
identity on the boundary, and gluing it to the identity on $X -
\nu(\Sigma) - \nu(\DT)$  to get a homology equivalence that is an
isomorphism on $\pi_1$.  To see that this gives the desired homotopy
equivalence $X - \nu(\Sigma_K) \to X -\nu(\Sigma)$, pass to the
universal cover of both manifolds.  By construction, there is a
relation $\mu_1\mu_2^d= 1$ between the meridians of the
configuration $\Sigma$.  From the proof of case~\ref{F1} of
Proposition~\ref{P:fundamentalgp}, this yields a relation $\mu_K
[S^1]^d = 1$, where $E(K) \times S^1$ is considered as a subset of $X
- \nu(\Sigma_K)$.

Let $g_1:\widetilde{E(K)} \to \widetilde{E(K)}$ generate the
covering transformations of the infinite cyclic cover
$\widetilde{E(K)} \to E(K)$, and let $g_2: \R \to \R$ be given by $t
\to t+1$.  Then the induced infinite-cyclic covering
$\widetilde{E(K) \times S^1} \to E(K) \times S^1$ is the quotient of
the $\Z\oplus \Z$ cover $\widetilde{E(K)} \times \R$ by the subgroup
of the covering group generated by $g_1g_2^d$. This is the same as
the mapping torus of $g_1^{-1}$, and hence (compare~\cite[Lemma
4.6]{kim-ruberman:complement}) is diffeomorphic to $E(K) \times \R$
which is homologically just a circle.  By a Mayer-Vietoris argument,
the map on universal covers is a homology equivalence, and hence the
original map is a homotopy equivalence.

As in~\cite[\S 2]{kim-ruberman:surfaces}, this map can be modified so that it is normally cobordant to the identity map, yielding a normal map $F: W^5 \to (X -\nu(\Sigma)) \times I$.  From~\cite{shaneson:product} and the vanishing of $\textrm{Wh}(\Z)$, the obstruction in $L^s_5(\Z)$ to surgering $W$ to obtain an s-cobordism may be computed as follows.  Choose a submanifold
$M^3 \subset X -\nu(\Sigma)$ carrying a generator of $H_3(X -\nu(\Sigma),\partial(X -\nu(\Sigma)))$, and assume that $F$ is transverse to $M \times I$.  Then $V = F^{-1}(M \times I)$ is an oriented $4$--manifold, and the surgery obstruction is given by $\frac18\textrm{sign}(V)$.

This obstruction may well be non-zero, but it can be killed by modifying the normal map $F$.  Choose an embedded $S^1 \times B^4$ in the interior of $W$, with $S^1 \times 0 \ti V$, carrying the $1$-dimensional homology, with the additional property that $F$ is a diffeomorphism of $S^1 \times B^4$ onto its image. There is a normal map
\[
S^1 \times \left(\#_n E_8\right)_0 \longrightarrow S^1 \times B^4
\]
where $n= \frac18\textrm{sign}(V)$, and  $E_8$ is a topological spin
manifold with signature $8$. The subscript $0$ denotes the punctured
manifold, and if $n<0$ the connected sum should be interpreted as a
sum of $E_8$ manifolds with the opposite orientation.   Replace $F$
on $S^1 \times B^4$ with this normal map, giving a normal map with
trivial surgery obstruction.  Surgery on this normal map produces
the desired $s$-cobordism, which is a product by Freedman's
theorem~\cite{freedman-quinn}.    As
in~\cite{kim-ruberman:surfaces}, the resulting homeomorphism between
the complements of $\Sigma$ and $\Sigma_K$ gives rise to a
topological ambient isotopy between $\Sigma$ and $\Sigma_K$.
\end{proof}

\section{Smoothly knotted configurations}\label{S:smooth}
In this section we describe a technique to show that modification of
a configuration by double point surgery can change its smooth
embedding type.  Let us assume that the
self-intersection $\Sigma \cdot \Sigma = n$ is greater than or equal
to $0$.  (The work of Mark~\cite{mark:hf-surfaces} would allow us to
consider configurations with negative self-intersection, but we do
not need this extension for the purposes of the current paper.)

Our idea is to replace a configuration $\Sigma$ by a smoothly embedded surface obtained by smoothing 
$\Sigma$ at its double points.  There is a small technical point that we must address, which is that the smoothing is not quite canonical; it depends on the relative orientation of the components.  
\begin{definition}\label{d:orientation} An orientation of a configuration $\Sigma$ is an orientation of each of its components.  If $\Sigma$ is oriented, then $-\Sigma$ denotes the configuration in which the orientation of each component has been reversed.
\end{definition}
The link of a given double point of $\Sigma$ is a Hopf link with
linking number $\pm 1$ with the sign given by the local intersection number. A
\emph{smoothing} of that double point is the replacement of the pair
of transversally intersecting disks by an oriented annulus $A$ that
spans the Hopf link and gives the same boundary orientation as that pair of disks.
Associated to an oriented configuration $(X,\Sigma)$, there is a canonically
associated smooth surface, which is the embedded surface resulting
from smoothing each double point in an oriented fashion. Let
$(X',\tSigma)$ be the result of blowing up $X$ at $n$ points on the
resulting smooth surface; note that $\tSigma \cdot \tSigma =0$.
We would like to show that any smooth invariant of the pair $(X',\tSigma)$ is {\em a fortiori} an
invariant of $(X,\Sigma)$ considered as a configuration up to diffeomorphism.  This is straightforward (and is stated as case (i) of the next lemma) if we insist that such a diffeomorphism preserve the orientation of the configuration, but requires a minor hypothesis and a little argument in the more general setting of Theorem~\ref{t:new-config}, where we make no assumptions about the orientation of the configuration. For simplicity, we give the hypotheses in the case that the configuration $\Sigma$ has two components. 
\begin{lemma}\label{L:smoothing}
Let $(X,\Sigma)$ and $(X,\Sigma')$ be oriented configurations, and suppose that $f:X \to X$ is an orientation-preserving diffeomorphism taking $\Sigma$ to $\Sigma'$.
\begin{enumerate}
\item If $f$ preserves the orientation of the configurations, then it extends to a diffeomorphism $f': (X',\tSigma) \to (X',\tSigma')$.
\item Suppose that $\Sigma = \Sigma_1 \cup \Sigma_2$, and that $\Sigma_1 \cdot \Sigma_2 \neq 0$. Then there is a diffeomorphism $(X',\tSigma) \to (X',\tSigma')$.
\end{enumerate}
In particular, if the homology classes of $\Sigma_1$ and $\Sigma_2$ are nontrivial, then $\tSigma$ is canonically associated to the equivalence class of $(X,\Sigma)$ up to {\em isotopy}. 
\end{lemma}
\begin{proof}
In both cases, a diffeomorphism between $(X,\Sigma)$ and $(X,\Sigma')$ must take double points to double points.  If, as in case (i), the local orientation of each disk is preserved, then the map $\Sigma \to \Sigma'$ extends over the annulus $A$ attached in the smoothing process.  The assumption in case (i) means that $\Sigma' \cdot \Sigma' = \Sigma \cdot \Sigma$, and we can assume that this diffeomorphism takes the $n$ points on $\Sigma$ where the blowup is done to the corresponding points on $\Sigma'$.  Thus the diffeomorphism extends (by the identity away from a ball) over the $n$ copies of $\CPbar^2$ added in the blowup.  This yields the diffeomorphism $f'$.  In case (ii),  we have that, as oriented surfaces, $f(\Sigma_1) = \pm \Sigma_1'$ and $f(\Sigma_2) = \pm \Sigma_2'$.  By invariance of intersection number, the signs must be the same for both components; case (i) covers the situation where both are positive.

If both signs are negative, then again the map $\Sigma \to \Sigma'$ extends over each annulus, and   $\Sigma' \cdot \Sigma' = \Sigma \cdot \Sigma$.  Recall that there is an orientation-preserving diffeomorphism (complex conjugation) on $\CP^2$ that reverses the orientation of $\CP^1$.  Gluing such a diffeomorphism to $f$ at each point where we perform a blowup gives a diffeomorphism $X' \to X'$ taking $\tSigma$ to $\tSigma'$. 

For the last part of the lemma, note that a diffeomorphism that is isotopic to the identity induces the identity map on homology.  Therefore, if $\Sigma_1$ and $\Sigma_2$ represent nontrivial homology classes, their orientations must be preserved by an isotopy, and so case (i) of the lemma applies.
\end{proof}
We now consider how a double point surgery on a configuration $(X,\Sigma)$ affects the smoothed surface constructed above.  From this point on, we assume that $\Sigma$ has two components, with nonzero intersection number. Hence, by Lemma~\ref{L:smoothing}, the smoothing $\tSigma$ is canonically associated to $\Sigma$.  

\begin{lemma}\label{L:double=rim} Consider a double point of the configuration $(X,\Sigma)$.  There is a smoothing  $(X,\Sigma')$ at this point such that the double point torus $\DT$ is disjoint from the annulus $A$. If $\alpha \subset A$ is an essential curve, then $\DT$ is the rim torus associated to $\alpha \subset \Sigma'$.
\end{lemma}
\begin{proof}
Consider a $4$-ball neighborhood of the double point, and recall that the double point torus may be chosen to live in
the boundary $3$-sphere, where it appears as the peripheral torus of either component of the Hopf link.  Make an
initial choice of the annulus as the standard Seifert surface for the Hopf link in this $3$-sphere. These surfaces are
labeled $\DT^{0}$ and $A^{0}$ in Figure~\ref{F:double=rim} below.  Note that their intersection is the curve $\alpha^{0}$
that is the core of the annulus.  Furthermore, note that each meridian $\mu_1$ for the first component of the Hopf link
intersects $A^{0}$ in a point, and that the corresponding meridian disk hits $A^{0}$ in an arc (drawn in blue).
\begin{figure}[!ht]\label{F:double=rim}
\centering
\psfrag{A}{$A^{0}$}
\psfrag{DT}{$\DT^{0}$}
\psfrag{D1}{$D_1$}
\psfrag{mu1}{$\mu_1$}
\psfrag{alpha}{$\alpha^{0}$}
\includegraphics[scale=.7]{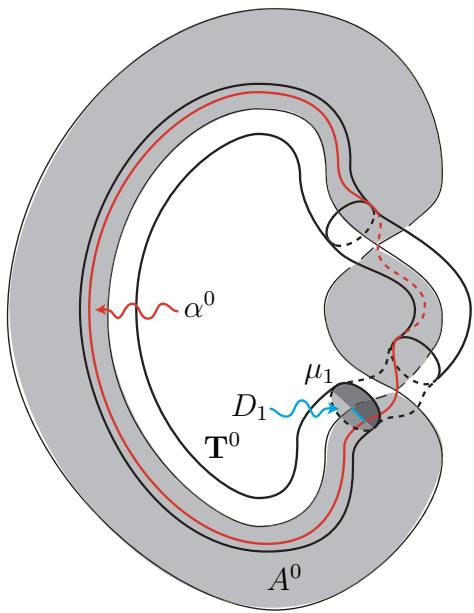}
\caption{Double point torus and annulus in $S^3$}
\end{figure}

Consider a function $f: A^{0}\to [1/2,1]$ that is $1$ on $\partial A^{0}$ and $1/2$ on $\alpha^{0}$, having no other critical
points.  Let $A$ be the result of pushing $A^{0}$ into the interior of the $4$-ball, leaving its boundary in the $3$-sphere,
so that a level set $f^{-1}(r)$ ends up at radius $r$ (in polar coordinates) in the $4$-ball.  Let $\DT$ be a copy of the
double point torus at radius $3/4$.    Each copy of the meridian $\mu_1$ on $\DT$ bounds a disk that intersects $A$
transversally in a single point; the union of those points forms a curve $\alpha$ (at radius $3/4$) parallel to
$\alpha^{0}$.  In this way, $\DT$ is identified with the boundary of normal bundle of $A$, restricted to $\alpha$.  By
definition, this is the rim torus.

The positions of $\DT$ and $A$ are illustrated schematically below; the figure on the left shows the positions in $S^3$, and the figure on the right shows the result of pushing $A^0$ and $\DT^0$ into the $4$-ball.
\begin{figure}[!ht]\label{F:double=rim-push}
\centering
\psfrag{A0}{$A^{0}$}
\psfrag{D1}{$D_1$}
\psfrag{mu0}{$\mu_1\subset \DT^{0}$}
\psfrag{mu}{$\mu_1\subset \DT$}
\psfrag{a0}{$\alpha^{0}$}
\psfrag{alpha}{$\alpha$}
\psfrag{to}{$\Longrightarrow$}
\psfrag{r1}{$1$}
\psfrag{s3}{$S^3$}
\psfrag{q}{$\frac34$}
\psfrag{h}{$r=\frac12$}

\includegraphics[scale=.9]{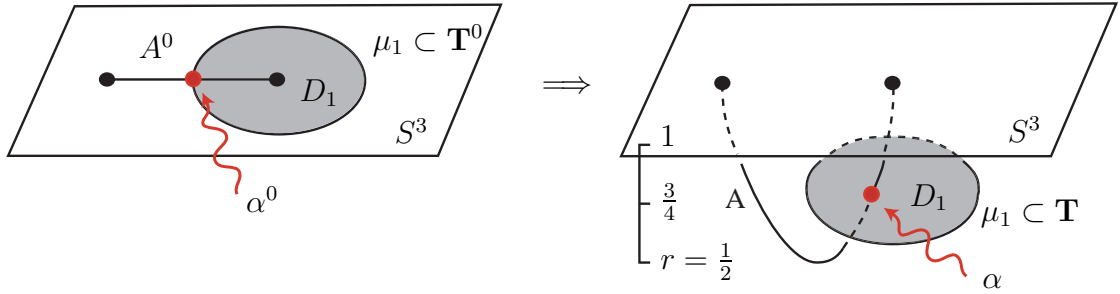}
\caption{Double point torus and rim torus in $B^4$}
\end{figure}

\end{proof}
Following~\cite{fs:addendum}, consider the set $\sT$ of $\spinc$
structures $\tau$ on $X' - (\tSigma \times D^2)$ whose restriction
to $\tSigma \times S^1$ is the pull-back of a $\spinc$ structure on
$\tSigma$ and which satisfies $\langle c_1(\tau), \tSigma\rangle =
\pm (2g(\tSigma)-2)$.  For any such $\spinc$ structure,
Kronheimer--Mrowka~\cite[Definition 3.6.4]{kronheimer-mrowka:monopole} define a
relative Seiberg--Witten invariant $\SW^\tau_{X',\tSigma}\in
\widehat{HM}(\tSigma \times S^1, \tau)$. This group is the monopole
Floer homology, and for $\tau$ with the properties described above
it is isomorphic to $\Z$.  Define $\ssw^\sT_{X,\Sigma}$ to be the
Laurent polynomial (cf.~\cite{fs:knots}) version of this relative
Seiberg--Witten invariant of $(X',\tSigma)$.

Parallel to the main theorem of~\cite{fs:surfaces}, as amended in~\cite{fs:addendum}, we have:
\begin{theorem}\label{T:smoothclass}  Suppose that $(X,\Sigma)$ is a configuration with $\Sigma_1 \cdot \Sigma_2 \neq 0$ and at least two intersections between $\Sigma_1$ and $\Sigma_2$, and let $\tau$ be a $\spinc$ structure such that $\ssw^\sT_{X,\Sigma} \neq 0$.  If $K_1$ and $K_2$ are two knots in $S^3$ and
if there is a diffeomorphism of pairs $f \co (X,\Sigma_{K_1}) \to (X,\Sigma_{K_2})$, then the set of coefficients (with
multiplicities) of $\Delta_{K_1}$ must be equal to that of $\Delta_{K_2}$.
\end{theorem}
\begin{proof}
The effect of $k$-twisted rim surgery on the relative invariant of an embedded surface is described in~\cite{fs:surfaces,fs:addendum}; the result is based on the formula~\cite[Theorem 1.5]{fs:knots} that describes the change in Seiberg-Witten invariants under knot surgery.   As remarked in~\cite{kim:surfaces}, this proof of this formula extends without change to the setting of $k$-twisted knot surgery.  This can be seen in various ways; the simplest is to use the skein-theoretic 
proof of~\cite[Theorem 1.5]{fs:knots}, and note that for the trivial knot, there is no difference between ordinary knot surgery and $k$-twisted knot surgery.

From Lemma~\ref{L:double=rim}, double point surgery $(X,\Sigma) \leadsto (X,\Sigma_{K,\phi})$ becomes a rim surgery $(X',\tSigma) \leadsto (X',\tSigma_{K,\phi})$ along the curve $\alpha \subset \tSigma$ created by smoothing the double point. If there is more than one intersection point, then $\alpha$ is non-separating on $\tSigma$, and hence (quoting~\cite{fs:addendum})
\[
\ssw^\sT_{X,\Sigma_{K,\phi}} = \ssw^\sT_{X',\tSigma_{K,\phi}} = \ssw^\sT_{X',\tSigma} \cdot\Delta_K(r^2) =
 \ssw^\sT_{X,\Sigma} \cdot\Delta_K(r^2)
\]
where $r$ is the homology class of the rim torus in $H_2(X - \tSigma)$.
The theorem follows as in~\cite{fs:addendum}.
\end{proof}
Suppose that $\Sigma$ is a configuration in a symplectic manifold
$X$.  We say that $\Sigma$ is {\em symplectic} if all of its
components are symplectically immersed surfaces, and all of the
double points are positive. Assuming that there are some intersection points, Lemma~\ref{L:smoothing} implies that the smoothing $\tSigma$ is canonically associated to $\Sigma$.  Then
by~\cite{gompf:symplectic,mccarthy-wolfson:blowup} the smoothing of
a symplectic configuration $(X,\Sigma)$ is a symplectic surface, and
we may apply Theorem~\ref{T:smoothclass}.  It is important to remark
that if $(X,\Sigma)$ is a symplectic configuration with only
positive double points, then there is a $\spinc$ structure $\tau$
with $\SW^\tau_{X,\Sigma} \neq 0$.    As
in~\cite{fs:surfaces,fs:addendum} this follows
from~\cite{taubes:sw-symplectic} by taking a fiber sum with an
appropriate K\"ahler manifold and then applying the gluing theory
of~\cite{kronheimer-mrowka:monopole}.

\subsection{Infinitely many distinct configurations}
We now have assembled the ingredients to prove the first of our main theorems.
\begin{proof}[of Theorem~\ref{t:new-config}]
Let $(X,\Sigma)$ be one of the configurations constructed in examples~\ref{E:nodal}, ~\ref{E:rational}, or  Proposition~\ref{P:tori}.  These are all symplectic configurations, and each of them can be used to prove the theorem, with slightly different arguments.  Let $K_r$, $r\in \N$, be a sequence of knots for which the Alexander polynomials have sets of coefficients that are pairwise distinct.

For each $r$, define $(X,\Sigma^{(r)})$ to be the $k$-twisted double point surgery $(X,\Sigma_{(K_r,A(k))})$ on $(X,K)$, where $k$ and the other parameters are chosen in the three cases as follows.
\begin{enumerate}
\renewcommand{\labelenumi}{(\roman{enumi})}
\item For $\Sigma$ as in Example~\ref{E:nodal}, choose $d_1 = 1$, and $k=0$.
\item For $\Sigma$ as in Example~\ref{E:rational}, choose $k = \pm 1$, and $p, q$ so that $q \geq 2$ and $(p \pm 1,q) =1$.
\item For $\Sigma$ as in Proposition~\ref{P:tori}, again choose $k = \pm 1$, and $m,n$ relatively prime.
\end{enumerate}
In each case Proposition~\ref{P:fundamentalgp} implies that the double point surgery preserves the fundamental group.

According to Theorem~\ref{T:smoothclass}, the configurations
$(X,\Sigma^{(r)})$ are smoothly pairwise distinct,
verifying~\ref{t:smooth} of the theorem.  To
verify~\ref{t:sigma1}, recall from Corollary~\ref{embedding} that
component $1$ of $\Sigma_{(K_r,A(k))}$ is embedded in $X$ as
$\Sigma_1\sharp A(K_r,k)$.   In case (i) above, $\Sigma_1$ is a copy
of $\CP^1 \subset \CP^2$.  According to the thesis of
P.~Melvin~\cite{melvin:thesis}, for $J$ a knotted $2$-sphere,
$(\CP^2, \CP^1 \sharp J)$ is equivalent to $(\CP^2,\CP^1)$ if and
only if the Gluck twist~\cite{gluck:twist} on $J$ yields $S^4$.  On the other hand, Gluck has shown~\cite{gluck:twist}
that this is the case for any $0$-twist spun knot.  It follows that
$\Sigma_1 \sharp A(K_r,0)$ is smoothly equivalent to $\Sigma_1$,
verifying~\ref{t:sigma1} of the theorem in the first case.  The
other two cases are easier, since by Zeeman's
theorem~\cite{zeeman:twist}, the knot $A(K_r,\pm 1)$ is the unknot
in $S^4$.   Item~\ref{t:sigma2} holds by construction; the
individual components $(X,\Sigma_2^{(r)})$ are simply equal to
$(X,\Sigma_2)$.     Finally, Theorem~\ref{T:topclass} implies that
$(X,\Sigma^{(r)})$ is topologically equivalent to $(X,\Sigma)$ for
all $r$, so that the last clause~\ref{t:top} of the theorem holds
as well.
\end{proof}
\section{Group Actions}\label{S:actions}
As shown by Fintushel-Stern-Sunukjian~\cite{fintushel-stern-sunukjian:actions}, twisted rim surgery, applied to the branch set of a cyclic branched cover, can give rise to smoothly exotic group actions on simply--connected $4$--manifolds.   The action (on the total space of the branched cover) is by definition semi-free.  In this section, we show that double point surgery gives rise to interesting group actions with somewhat more complicated singular sets consisting of a configuration of surfaces.  The covering group will be of the form $\zm \oplus \zn$, with the fixed points of the $\zm$ subgroup being one component of the configuration, and those of the $\zn$ subgroup being another component.

The local picture at an intersection point in the configuration may be described as follows; compare~\cite{montesinos:twins.I}. Let $\zeta_k$ denote a primitive $k^{th}$ root of unity.  Then $\zm \oplus \zn$ acts on $\C^2$ with $\zm$ acting on the first $\C$ factor by $\zeta_m$, and $\zn$ acting on the second factor by $\zeta_n$.  The action restricts to an action on the unit ball in $\C^2$, with singular set the unit discs in the two factors, and with fixed point set the origin.  The quotient is also a $4$-ball, with the image of the singular set again a pair of standardly embedded discs meeting at the origin.  We will refer to the image of the singular set in the quotient as the branch set, and refer to any $\zm \oplus \zn$ action with this local structure at fixed points as having standard type.

Conversely, suppose that we have a configuration of surfaces $\Sigma$ in $X$ and an epimorphism
$\phi: \pi_1(X - \Sigma) \to \zm \oplus \zn$ with the property that each intersection point,
$\phi$ takes the meridian of one of the discs to a generator of $\zm$ and the other meridian to a generator
of $\zn$.  Then we can form a regular branched covering $\widetilde{X} \to X$ with covering group $\zm \oplus \zn$.  The branch set is $\Sigma$, written as a union $\Sigma(m) \cup \Sigma(n)$ (note that these may themselves be unions of components), and the local model near each intersection point of $\Sigma(m)$ and $\Sigma(n)$ will be as described in the previous paragraph.

Applying a double point surgery to an appropriately chosen configuration, we will show the following.
\begin{theorem}\label{T:exotic-action}
Let $m$ and $n$ be relatively prime. There is a simply--connected $4$--manifold
$\widetilde{X}$ supporting infinitely many $\zm \oplus \zn$-actions of
standard type that are smoothly inequivalent, but topologically
equivalent.  
\end{theorem}
\begin{proof}
We have seen in Proposition~\ref{P:tori} that there is a two-component configuration $\Sigma$ of surfaces in a simply--connected manifold $X$ such that $\pi_1(X-\Sigma)$ is isomorphic to $\zm \oplus \zn$ in such a way that the meridian of $\Sigma_1$ is a generator of
$\zm$ and the meridian of $\Sigma_2$ is a generator of $\zn$.
Consider the regular branched $\zm \oplus \zn$-cover $\widetilde{X}$ of $X$ over
$\Sigma$ with the properties that the branch set equals $\Sigma_1 \cup
\Sigma_2$, the branching index along $\Sigma_1$ is $m$ and the branching index along $\Sigma_2$
is $n$. Observe that the map $\widetilde{X} \to X$ may be constructed in two stages, as a composition of branched coverings $p_n\circ q_m$ as indicated below.
\begin{equation}\label{br-diagram}
\wtX \overset{q_m}{\longrightarrow} \wtX^n \overset{p_n}{\longrightarrow}  X
\end{equation}
The $\zn$-cover $p_n:\wtX^{n} \to X$ branched over $\Sigma_2$
corresponds to the epimorphism $\pi_1(X-\Sigma_2)\to \zn$ that maps
the meridian $\mu_2$ to $1$.  The  $\zm$-cover $q_m: \wtX \to \wtX^n$ is
branched over the preimage $p_n^{-1}(\Sigma_1)$ in  $\wtX^n$, and corresponds to the epimorphism 
$\pi_1(\wtX^n - p_n^{-1}(\Sigma_1))\to \pi_1(X -\Sigma_1) \to \zm$.

Now, we perform a $k$-twisted double point surgery on the
configuration $\Sigma$, giving a new configuration $(X,\Sigma_K)$. Let $\widetilde{X}_K$ be the $\zm \oplus \zn$-cover of $X$ branched over $\Sigma_K$; it is also constructed as a composition of branched covers as in~\eqref{br-diagram}.   Choosing, as in case (iii) of Theorem~\ref{t:new-config}, an infinite collection of knots $K$ so that the configurations $\Sigma_K$ are smoothly distinct, we have that the $\zm \oplus
\zn$ actions on the manifolds  $\widetilde{X}_K$ will be smoothly distinct.  So the proof of Theorem~\ref{T:exotic-action} will be completed by the following:\\[1ex]
{\bf Claim:} If $k$ and $m$ are relatively prime, then $\widetilde{X}_K$ is
diffeomorphic to $\widetilde{X}$.\\[1ex]
{\bf Proof of claim.}
Following the construction above, $\wtX_K$ is the $m$-fold branched cover of $\wtX_K^n$, branched along $p_n^{-1}(\Sigma_{1,K})$.  So our main task is to understand that preimage.  Recall that the $k$-twisted double point surgery is given by
\[
(X,\Sigma_K)=(X,\Sigma)-\DT\times D^2\cup_{\phi} E(K) \times
S^1,
\] 
where the gluing map $\phi$ is given by Equation~\eqref{twistbleptsurgery}.

Viewing $E(K) \times S^1$ as a subset of of $X - \Sigma_K$, note that the map $\pi_1(E(K)
\times S^1) \to\zn$ corresponding to the covering $\wtX^n_K \to X$ takes $[S^1]$ to a generator of $\zn$ and $\mu_K$
to the trivial element. Thus, we obtain $(\wtX_K^{n},p_n^{-1}(\Sigma_{K}))$ from $(\wtX^n,p_n^{-1}(\Sigma))$ by a
double point surgery along the lifted double point torus
$\widetilde\DT=\tilde\mu_1\times \tilde\mu_2$ in $\wtX^n$:
\[
\wtX_K^{n}=\wtX^n-\widetilde\DT\times D^2\cup_{\tilde\phi}E(K) \times
S^1,
\] where 
\[
\tilde\phi(\tilde\mu_1)=\mu_{K},\ 
\tilde\phi(\tilde\mu_2)=[S]^1+k\mu_{K}, \ \mathrm{and}\  \tilde\phi(\p
D^2)=\lambda_{K}.
\]
This operation is exactly the $k$-twisted double
point surgery on the configuration
$(\wtX^n,p_n^{-1}(\Sigma))$. The local nature of double point surgery (cf.~\eqref{dbleptsurgery} and
Proposition~\ref{P:embedd-plotnick}) allow us to write $\wtX_K^{n}$
as follows. Given the matrix $A(k)$ in~\eqref{E:matrix2},
\begin{equation}\label{E:n-bcover}
(\wtX_K^{n},\tSigma_K)=(\wtX^n, p_n^{-1}(\Sigma_1)\cup p_n^{-1}(\Sigma_2))\sharp (S^4, (S_1\cup S_2)_{K, A(k)}).
\end{equation}

Recall that cutting and regluing using the matrix $A(k)$ changes the twin $S_1 \cup S_2$ into the $k$-twist twin $(S_1\cup S_2)_{K,
A(k)}$. Hence by Corollary~\ref{embedding}, one component of $p_n^{-1}(\Sigma)$
is $p_n^{-1}(\Sigma_1)\sharp A(K,k)$ and the other one remains
$p_n^{-1}(\Sigma_2)$.

From this discussion, the $\zm \oplus \zn$-cover $\widetilde{X}_K$ is the $\zm$-cover of $\wtX_K^{n}$ branched over $p_n^{-1}(\Sigma_1)\sharp A(K,k)$.  The branched cover decomposes as the connected sum of the $m$-fold branched cover of $\wtX^n$ along $p_n^{-1}(\Sigma_1)$ and the $m$-fold branched cover of $S^4$ over the $k$-twist spun knot  $A(K,k)$, denoted by $(S^4, A(K,k))^m$.  By Corollary 6.1 in~\cite{plotnick:fibered}, if $m$ and $k$ are relatively prime, then $(S^4, A(K,k))^m$ is smoothly $S^4$. Hence, $\widetilde{X}_K$ is diffeomorphic to $\wtX$.
\end{proof}

 \end{document}